\documentclass[12pt,twocolumn]{article}
\usepackage{amscd, amsmath, amssymb, amsfonts, amsthm}
\usepackage{euscript}
\usepackage{graphicx}
\usepackage[square, comma, sort&compress, numbers]{natbib}
\usepackage[colorlinks,linkcolor=blue,citecolor=magenta,anchorcolor=red, bookmarksopen=false]{hyperref}

\def\pref#1{(\ref{#1})}

\theoremstyle{plain}
\newtheorem{prop}{Proposition}[section]
\newtheorem{thm}[prop]{Theorem}
\newtheorem{lem}[prop]{Lemma}
\newtheorem{cor}[prop]{Corollary}
\newtheorem{ques}[prop]{Question}

\theoremstyle{definition}

\newtheorem{rem}[prop]{Remark}

\def\V{\mathrm{V}}
\def\E{\mathrm{E}}
\def\ifof{if and only if }
\def\f#1#2{\frac{#1}{#2}}
\def\se{\subseteq}
\def\N{\mathrm{N}}
\def\b#1{\overline{#1}}
\def\sm{\setminus }
\newcommand{\tohi}{\emptyset}
\newcommand{\give}{$\Rightarrow$}
\newcommand{\rgive}{$\Leftarrow$}
\renewcommand{\l}{\left}
\renewcommand{\r}{\right}
\newcommand{\link}{\mathrm{link}}
\newcommand{\core}{\mathrm{core}}
\newcommand{\Q}{\mathbb{Q}}

\begin{document}
\date{}
\title{Some Combinatorial Characterizations of Gorenstein Graphs with Independence Number Less than Four}
\author{Mohammad Reza Oboudi, Ashkan Nikseresht\footnote{Corresponding author}\\
\it\small Department of Mathematics, College of Sciences,  Shiraz University \\
\it\small 71457-44776, Shiraz, Iran\\
\it\small E-mail: mr\_oboudi@yahoo.com  and mr\_oboudi@shirazu.ac.ir\\
\it\small E-mail: a\_nikseresht@shirazu.ac.ir
 }
 \maketitle

\begin{abstract}
Let $\alpha=\alpha(G)$ be the independence number of a simple graph $G$ with $n$ vertices and $I(G)$ be its edge
ideal in $S=K[x_1,\ldots, x_n]$. If $S/I(G)$ is Gorenstein, the graph $G$ is called Gorenstein over $K$ and if $G$ is
Gorenstein over every field, then we simply say that $G$ is Gorenstein. In this article, first we state a condition
equivalent to $G$ being Gorenstein and using this we give a characterization of Gorenstein graphs with $\alpha=2$.
Then we present some properties of Gorenstein graphs with $\alpha=3$ and as an application of these results we
characterize triangle-free Gorenstein graphs with $\alpha=3$.
\end{abstract}


                                        \section{Introduction}

Throughout this paper, $K$ is a field, $S=K[x_1,\ldots, x_n]$ and $G$ denotes a simple undirected graph with vertex set
$\V(G)=\{v_1,\ldots, v_n\}$ and edge set $\E(G)$. Recall that the \emph{edge ideal} $I(G)$ of $G$ is the ideal of $S$
generated by $\{x_ix_j|v_iv_j\in \E(G)\}$. Many researchers have studied how algebraic properties of $S/I(G)$ relates
to the combinatorial properties of $G$ (see \cite{hibi, large girth, tri-free, planar goren} and references therein).
An important algebraic property that recently has gained attention is being Gorenstein. Recall that Gorenstein rings
and Cohen-Macaulay rings (CM rings for short) are central concepts in commutative algebra. We refer the reader to
\cite{CM ring} for their definitions and basic properties. We say that $G$ is a \emph{Gorenstein} (resp.
\emph{Cohen-Macaulay}) graph over $K$, if $S/I(G)$ is a Gorenstein (resp. CM) ring. When $G$ is Gorenstein (resp. CM)
over every field, we say that $G$ is Gorenstein (resp. CM). In \cite{large girth} and \cite{tri-free} characterizations
of planar Gorenstein graphs of girth at least four and triangle-free Gorenstein graphs are presented, respectively.
(Recall that $G$ is said to be \emph{triangle-free} when no subgraph of $G$ is a triangle). Also in \cite{planar goren}
a condition on a planar graph equivalent to being Gorenstein is stated and proved. An importance of characterizing
Gorenstein graphs comes from the Charney-Davis conjecture which states that ``flag simplicial complexes'' which are
Gorenstein over $\Q$ satisfy a certain condition (see \cite[Problem 4]{frontier}). In \cite{frontier}, Richard P.
Stanley mentioned this conjecture as one of the ``outstanding open problems in algebraic combinatorics'' at the start
of the 21$^\mathrm{st}$ century. An approach to solve this conjecture is trying to give a characterization of
Gorenstein ``flag simplicial complexes'' which is equivalent to a characterization of Gorenstein graphs by \cite[Lemma
9.1.3]{hibi}.

Denote by $\alpha(G)$ the \emph{independence number} of $G$, that is, the maximum size of an independent set of $G$. It
is well-known that if $\alpha(G)=1$ (that is, $G$ is complete) then $G$ is Gorenstein \ifof $G=K_1$ or $G=K_2$ (where
$K_n$ denotes the complete graph on $n$ vertices). Moreover, it is easy to see  that if $\alpha(G)=2$, then $G$ is
Gorenstein \ifof $G$ is the complement of a cycle of length at least four (see Corollary \ref{goren alpha=2} below).
Therefore, it is natural to ask ``if $\alpha(G)=3$, then when is $G$ Gorenstein?'' In this paper, we present some
properties of Gorenstein graphs with $\alpha=3$ and find all Gorenstein graphs with $\alpha=3$ which are also
triangle-free. Before stating our main results and since we need some results on Gorenstein simplicial complexes, first
we briefly review simplicial complexes and their Stanley-Reisner ideal.

 Recall that a \emph{simplicial complex} $\Delta$ on the vertex set $V= \V(\Delta)=
\{v_1,\ldots, v_n\}$ is a family of subsets of $V$ (called \emph{faces}) with the property that $\{v_i\}\in \Delta$ for
each $i\in [n]=\{1,\ldots,n\}$ and if $A\se B\in \Delta$, then $A\in \Delta$. In the sequel, $\Delta$ always denotes a
simplicial complex. Thus the family $\Delta(G)$ of all cliques of a graph $G$ is a simplicial complex called the
\emph{clique complex} of $G$. Also $\Delta(\b G)$ is called the \emph{independence complex} of $G$, where $\b G$
denotes the complement of $G$. Note that the elements of $\Delta(\b G)$ are independent sets of $G$. The ideal of $S$
generated by $\{\prod_{v_i\in F} x_i|F\se V$ is a non-face of $\Delta\}$ is called the \emph{Stanley-Reisner ideal} of
$\Delta$ and is denoted by $I_\Delta$ and $S/I_\Delta$ is called the \emph{Stanley-Reisner algebra} of $\Delta$ over
$K$. Therefore we have $I_{\Delta(\b G)}= I(G)$. If the Stanley-Reisner algebra of $\Delta$ over $K$ is Gorenstein
(resp. CM), then $\Delta$ is called Gorenstein (resp. CM) over $K$. The relation between combinatorial properties of
$\Delta$ and algebraic properties of $S/I_\Delta$ is well-studied, see for example \cite{symbolic, hibi,stanley, our
chordal, my vdec} and the references therein.

The dimension of a face $F$ of $\Delta$ and the simplicial complex $\Delta$ are defined to be $|F|-1$ and
$\max\{\dim(F)|F\in \Delta\}$, respectively. Recall that $\alpha(G)$ denotes the \emph{independence number} of $G$,
that is, $\alpha(G)=\dim \Delta(\b G)+1$. A graph $G$ is called \emph{well-covered}, if all maximal independent sets of
$G$ have size $\alpha(G)$ and we say that $G$ is a \emph{W$_2$ graph}, if $|\V(G)|\geq 2$ and every pair of disjoint
independent sets of $G$ are contained in two disjoint maximum independent sets. In some texts, W$_2$ graphs are called
1-well-covered graphs. The reader can see \cite{W2} for more on W$_2$ graphs. In the following lemma, we have collected
two known results on W$_2$ graphs and their relation to Gorenstein graphs.

\begin{lem}\label{W2}
\begin{enumerate}
\item \label{Goren=> W2} (\cite[Lemma 3.1]{large girth} or \cite[Lemma 3.5]{tri-free}) If $G$ is a graph without
    isolated vertices and $G$ is Gorenstein over some field $K$, then $G$ is a W$_2$ graph.

\item \label{tri-free Goren} (\cite[Proposition 3.7]{tri-free}) If $G$ is triangle-free  and without isolated
    vertices, then $G$ is Gorenstein \ifof $G$ is W$_2$.
%

\end{enumerate}
\end{lem}

Let $f_i$ be the number of $i$-dimensional faces of $\Delta$ (if $\Delta\neq \tohi$, we set $f_{-1}=1$), then $(f_{-1},
\ldots, f_{d-1})$ is called the \emph{$f$-vector} of $\Delta$, where $d-1=\dim(\Delta)$. Now define $h_i$'s such that $
h(t)= \sum_{i=0}^d h_it^i= \sum_{i=0}^d f_{i-1}t^i(1-t)^{d-i}$. Then $h(t)$ is called the \emph{$h$-polynomial} of
$\Delta$. Recall that the polynomial $I(G,x)=\sum_{i=0}^{\alpha(G)} a_ix^i$, where $a_i$ is the number of independent
sets of size $i$ in $G$ and $a_0=1$, is called the \emph{independence polynomial} of $G$. Note that $(a_0,a_1,\ldots,
a_{\alpha(G)})$ is the $f$-vector of $\Delta(\b G)$. There are many papers related to this polynomial in the
literature, see for example \cite{survey}.

For $F\in \Delta$, let $\link_\Delta(F)=\{A\sm F| F\se A\in \Delta\}$. If all maximal faces of $\Delta$ have the same
dimension and for each $F\in \Delta$, we have $\sum_{i=-1}^{d_F} (-1)^i f_i(F)= (-1)^{d_F}\quad (*)$, where $d_F= \dim
\link_\Delta(F)$ and $(f_i(F))_{i=-1}^{d_F}$ is the $f$-vector of $\link_\Delta(F)$, then $\Delta$ is said to be an
\emph{Euler complex}. Another complex constructed from $\Delta$ is $\core(\Delta)$. The set of vertices of
$\core(\Delta)$ is $V_C=\{v\in \V(\Delta)|\exists G\in \Delta \text{ such that } G\cup\{v\}\notin \Delta\}$ and
$\core(\Delta)= \{F\in \Delta|F\se V_C\}$.
                                        \section{Main results}

First we establish a theorem which presents a condition on graphs equivalent to being Gorenstein. For this, we need the
following lemma.
\begin{lem}\label{lem1}
If $G$ is a graph without any isolated vertex and Gorenstein over some field $K$, then $I(G,-1)=(-1)^{\alpha(G)}$.
\end{lem}
\begin{proof}
Let $\Delta=\Delta(\b G)$ and $d=\dim \Delta=\alpha(G)-1$. If $v\in \V(G)$, then as $v$ is not isolated, there is a
$u\in \V(G)$ such that $uv\in \E(G)$. Hence $G=\{u\}\in \Delta$ and $G\cup\{v\} \notin \Delta$. This means that $v\in
\core(\Delta)$ and hence $\core(\Delta)=\Delta$. So according to \cite[Theorem 5.5.2]{CM ring}, $\Delta$ is an Euler
complex. In particular, noting that $\link_\Delta(\tohi)= \Delta$, if we write $(*)$ with $F=\tohi$, we get
$\sum_{i=-1}^{d} (-1)^i f_i= (-1)^{d}$ where $(f_i)_{i=-1}^d$ is the $f$-vector of $\Delta$. Thus
\begin{align}
I(G,-1)& = \sum_{i=0}^{\alpha(G)} a_i(-1)^i= \sum_{i=0}^{\alpha(G)}(-1)^i f_{i-1} \notag \\
& = -\sum_{j=-1}^{d} (-1)^j f_j  \notag\\
& = (-1)^{d+1}=(-1)^{\alpha(G)}. \notag
\end{align}
\end{proof}
\begin{rem}\label{no isol => Euler}
The above argument shows that if $G$ has no isolated vertex, then $\Delta(\b G)=\core(\Delta(\b G))$ and if moreover
$G$ is Gorenstein over some field, then $\Delta(\b G)$ is an Euler complex.
\end{rem}

Suppose that $F\se \V(G)$. Here by $\N[F]$ we mean $F\cup \{v\in \V(G)| uv\in \E(G)$ for some $u\in F\}$ and we set
$G_F=G\sm \N[F]$.  In particular, for $F=\tohi$ we have $\N[F]=\tohi$ and $G_F=G$. Our next theorem gives an equivalent
condition on graphs with $\alpha\geq 2$ for being Gorenstein. Recall that if $\alpha(G)=1$, then $G$ is Gorenstein
\ifof $G$ is isomorphic to $K_1$ or $K_2$. Note that if $\Delta$ is a simplicial complex with dimension at most one,
then we can view $\Delta$ as a graph on vertex set $\V(\Delta)$ by considering 1-dimensional faces of $\Delta$ as
edges.

\begin{thm}\label{goren graphs}
Suppose that $G$ is a graph without any isolated vertex and $\alpha(G)\geq 2$. Then $G$ is Gorenstein over $K$ \ifof
all of the following conditions hold:
\begin{enumerate}
\item \label{goren graphs 1} $G$ is CM over $K$;

\item \label{goren graphs 2} $I(G,-1)= (-1)^{\alpha(G)}$;

\item \label{goren graphs 3} For each independent set $F$ of $G$ with size $\alpha(G)-2$, the graph $\b {G_F}$ is a
    cycle of length at least $4$.
\end{enumerate}
\end{thm}
\begin{proof}
(\give): By Lemma \ref{lem1}, $I(G,-1)=(-1)^{\alpha(G)}$ and $G$ is CM over $K$ by definition. If $G=K_2$, then it has
no independent set of size 2 and hence \pref{goren graphs 3} holds trivially. So we assume that $G\neq K_2$.

Suppose that $F$ is an independent set of $G$ with size $\alpha(G)-2$. First we claim that $\b {G_F}$ is not a triangle
or a path of length 1 or 2. Suppose that the claim is not true. First, assume that $\b {G_F}$ consists of just an edge
$ab$. Let $A=F\cup\{a\}$ and $B=\{b\}$ which are two disjoint independent sets of $G$. By Lemma \ref{W2}\pref{Goren=>
W2}, $G$ is W$_2$ and there exist independent sets $A_0$ and $B_0$ of size $\alpha(G)$ such that $A\se A_0$, $B\se B_0$
and $A_0\cap B_0=\tohi$. But the only vertex of $G$ which is not adjacent to any vertex of $A$ in $G$ is $b$, so
$A_0=A\cup\{b\}$. But $b\in B\se B_0$ which contradicts $A_0\cap B_0=\tohi$. From this contradiction it follows that
$\b {G_F}$ is not a path of length 1.

Now assume that $\b {G_F}$ is a path of length 2 with edges $ab$ and $bc$. Let $A=\{a,b\}$ and $B=F\cup \{c\}$. Again
$A$ and $B$ are disjoint independent sets of $G$ and since $G$ is W$_2$, there exist independent sets $A_0$ and $B_0$
of size $\alpha(G)$ such that $A\se A_0$, $B\se B_0$ and $A_0\cap B_0=\tohi$. But $b$ is the only vertex of $G$ not
adjacent to any vertex of $B$. Therefore $b\in A_0\cap B_0$, which is a contradiction, thus $\b {G_F}$ is not a path of
length 2. Also we have $\alpha(G_F)=2$ and hence $\b{G_F}$ can not be a triangle. This concludes the proof of the
claim.

Note that because $\Delta=\Delta(\b G)$ contains all vertices of $G$, it is not empty and since $G\neq K_1,K_2$, its
independence complex $\Delta$ does not consist of only one or two isolated vertices. Also by Remark \ref{no isol =>
Euler}, $\core(\Delta)= \Delta$. Hence by equivalence of parts (a) and (e) of \cite[Chapter II, Theorem 5.1]{stanley},
$\link_\Delta(F)$ is either a cycle or a path of length 1 or 2 (when viewed as a graph, as mentioned in the notes
before this theorem). Let $A\in \link_\Delta(F)$, then there is a $B\in \Delta$ containing $F$, such that $A=B\sm F$.
So $B=A\cup F$ is an independent set of $G$. Hence $A$ is an independent set of $G_F$, that is, $A\in \Delta(\b{G_F})$.
Conversely, if $A\in \Delta(\b{G_F})$, then $A\cup F$ is an independent set of $G$ and hence $A\in \link_\Delta(F)$.
Therefore, $\link_\Delta(F)=\Delta(\b{G_F})$. If we view this 1-dimensional simplicial complex as a graph, then
$\Delta(\b{G_F})= \b{G_F}$. Hence $\b{G_F}$ is either a cycle or a path of length 1 or 2. But as we showed above,
$\b{G_F}$ can not be a triangle or a path of length 1 or 2. Consequently, $\b{G_F}$ is a cycle of length at least 4, as
required.

(\rgive): Let $F$ be an independent set of size $\alpha(G)-2$ and $\Delta=\Delta(\b G)$. According to the argument in
the above paragraph, $\link_\Delta(F)=\Delta(\b{G_F})=\b{G_F}$. Therefore, $\link_\Delta(F)$ is a cycle for each
independent set of size $\alpha(G)-2$ of $G$. Also $\core(\Delta)=\Delta$ by Remark \ref{no isol => Euler}. Hence the
result follows from the equivalence of (a) and (e) of \cite[Chapter II, Theorem 5.1]{stanley}.
\end{proof}

\begin{rem}
Note that Theorem \ref{goren graphs} does not classify all Gorenstein graphs with $\alpha \geq 2$. In fact, for this we
first need to classify CM graphs, which is a hard task.
\end{rem}

As an immediate corollary of the previous theorem, we get the following characterization of Gorenstein graphs with
$\alpha(G)=2$.
\begin{cor}\label{goren alpha=2}
Suppose that $G$ is a graph without any isolated vertex and $\alpha(G)=2$. Then the following are equivalent.
\begin{enumerate}
\item \label{goren alpha=2 1} $G$ is Gorenstein over some field.
\item \label{goren alpha=2 2} $G$ is Gorenstein over every field.
\item \label{goren alpha=2 3} $G$ is the complement of a cycle of length at least four.
\end{enumerate}
\end{cor}
\begin{proof}
\pref{goren alpha=2 2} \give\ \pref{goren alpha=2 1} is trivial and \pref{goren alpha=2 1} \give\ \pref{goren alpha=2
3} follows from Theorem \ref{goren graphs}\pref{goren graphs 3} with $F=\tohi$.

\pref{goren alpha=2 3} \give\ \pref{goren alpha=2 2}: Suppose that $G$ is the complement of an $n$-cycle with $n\geq
4$. If we view $\Delta(\b G)$ as a graph, then $\Delta(\b G)$ is a cycle and hence is connected. Therefore, by
\cite[Exercise 5.1.26(c)]{CM ring}, $\Delta(\b G)$ or equivalently $G$ is CM over every field. Also $I(G,x)= 1+nx+nx^2$
and $I(G,-1)=1= (-1)^{\alpha(G)}$. Thus the result follows from Theorem \ref{goren graphs}.
\end{proof}

\begin{rem}
Note that item \eqref{goren alpha=2 3} of the above corollary determines a large class of Gorenstein graphs
combinatorially.
\end{rem}

Next we find some properties of Gorenstein graphs with $\alpha(G)=3$. For this, we need a lemma.
\begin{lem}\label{lem2}
Assume that $G$ is a graph without any isolated vertex and Gorenstein over a field $K$ and suppose that $\alpha(G)$ is
odd, then $I(G,-\f{1}{2})=0$.
\end{lem}
\begin{proof}
Let $\Delta=\Delta(\b G)$ and $\alpha=\alpha(G)$. By Remark \ref{no isol => Euler}, $\Delta$ is an Euler complex of
dimension $\alpha-1$. Thus by \cite[Theorem 5.4.2]{CM ring}, we have $h_i=h_{\alpha-i}$, where $h(t)= \sum_{i=0}^\alpha
h_it^i= \sum_{i=0}^\alpha f_{i-1}t^i(1-t)^{\alpha-i}$ is the $h$-polynomial of $\Delta$. Since $\alpha$ is odd, this
means that $h(-1)=0$. Now
\begin{align}
I\l(G,-\f{1}{2} \r)& = \sum_{i=0}^\alpha a_i \l(-\f{1}{2} \r)^i \notag\\
& =\sum_{i=0}^\alpha f_{i-1} \l(-\f{1}{2} \r)^i\notag  \\
 &= \f{1}{2^\alpha} \sum_{i=0}^\alpha f_{i-1} (-1)^i(1-(-1))^{\alpha-i} \notag \\
 & = \f{1}{2^\alpha} h(-1)=0. \notag
\end{align}
\end{proof}

Now we can prove a theorem which determines the number of edges and the independence polynomial of Gorenstein graphs
with independence number three.
\begin{thm}\label{number of edges}
Suppose that $G$ is a graph with $n$ vertices and $m$ edges and without any isolated vertex. Also assume that $G$ is
Gorenstein over some field and $\alpha(G)=3$. Then the following hold.
\begin{enumerate}
\item \label{number of edges 1} $m=\f{n^2-7n+12}{2}$.
\item \label{number of edges 2} $G$ has exactly $2n-4$ independent sets of size three and exactly $3n-6$ independent
    sets of size two.
\item \label{number of edges 3} $I(G,x)=1+nx+(3n-6)x^2+(2n-4)x^3$.
\end{enumerate}
\end{thm}
\begin{proof}
Suppose that $I(G,x)=\sum_{i=0}^{\alpha(G)} a_i x^i$ is the independence polynomial of $G$. Then $a_0=1$, $a_1=n$ and
$a_2$ and $a_3$ are the number of independent sets of size two and three of $G$, respectively. According to Lemmas
\ref{lem1} and \ref{lem2}, $I(G,-1)= (-1)^3=-1$ and $I(G,-\f{1}{2})= 0$. It follows that $a_2-a_3=n-2$ and $2a_2-a_3=
4n-8$. Therefore $a_2=3n-6$ and $a_3=2n-4$, which proves \pref{number of edges 2} and \pref{number of edges 3}. On the
other hand, independent sets of size two are exactly those 2-subsets $\{u,v\}$ of $\V(G)$ such that $uv\notin \E(G)$.
Thus $3n-6= a_2={n \choose 2} - m$ and $m=\f{n^2-7n+12}{2}$.
\end{proof}

In what follows, we apply the previous results to find all triangle-free Gorenstein graphs with $\alpha=3$. It should
be mentioned that this result also follows by using \cite[Theorem 4.4]{tri-free} and \cite[Theorem 3.8]{symbolic}, but
the proof of these theorems use much algebraic background, against the following graph theoretic proof. In what
follows, by $G_v$ we mean $G_{\{v\}}=G\sm \N[v]$, where $v\in \V(G)$. Also $C_n$ denotes the $n$-vertex cycle.
\begin{thm}\label{tri-free-ind3}
Suppose that $G$ is triangle-free, $\alpha(G)=3$ and $G$ has no isolated vertices. Then the following are equivalent:
\begin{enumerate}
\item \label{tri-free-ind3 1} $G$ is Gorenstein over every field;

\item \label{tri-free-ind3 2} $G$ is Gorenstein over some field;

\item \label{tri-free-ind3 3} $G$ is W$_2$.

\item \label{tri-free-ind3 4} $G$ is isomorphic to one of the graphs depicted in Figure \ref{fig}.
\end{enumerate}
\end{thm}
\begin{figure*}
\begin{center}
\includegraphics[scale=0.9]{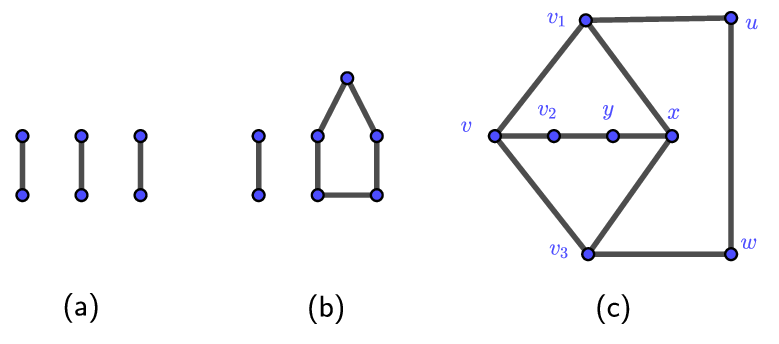}
\end{center}
\caption{Triangle-free Gorenstein graphs without isolated vertex and with $\alpha=3$} \label{fig}
\end{figure*}

\begin{proof}
\pref{tri-free-ind3 1} $\iff$ \pref{tri-free-ind3 3} is an especial case of \cite[Proposition 3.7]{tri-free} and
\pref{tri-free-ind3 1} \give\ \pref{tri-free-ind3 2} is trivial.

\pref{tri-free-ind3 4} \give\ \pref{tri-free-ind3 3}: Let $G$ be one of the graphs in Figure \ref{fig}. For each $v\in
\V(G)$, one can check that $G\sm v$ is well-covered and $\alpha(G\sm v)=\alpha(G)$. According to \cite[Theorem 1]{W2},
this means that $G$ is W$_2$.

\pref{tri-free-ind3 2} \give\ \pref{tri-free-ind3 4}: Suppose that $n=|\V(G)|$ and $m=|\E(G)|$ and let $v\in \V(G)$.
According to Theorem \ref{goren graphs}\pref{goren graphs 3}, $G_v$ is the complement of a cycle on at least 4
vertices. Since complement of any cycle with at least 6 vertices has a triangle, thus $G_v$ is isomorphic to $\b{C_4}$
or $\b{C_5}$. In particular, $n-1-\deg(v)= |\V(G_v)|\in \{4,5\}$, that is, either $\deg(v)=n-5$ or $\deg(v)=n-6$.
Assume that $a$ and $b$ are the number of vertices of $G$ with degree $n-5$ and $n-6$, respectively. Therefore $a+b=n$
and $(n-5)a+(n-6)b=2m=n^2-7n+12$ by Proposition \ref{number of edges}. Solving these two equations for $a$ and $b$, we
get $a=12-n$ and $b=2n-12$. In particular, it follows that $6\leq n\leq 12$.

Suppose that $n\geq 7$. Then as $b>0$, there exists a vertex of $G$, say $v$,  with degree $n-6$. Let $A=\N(v)$ and
$B=\V(G)\sm (\N[v])$. By Theorem \ref{goren graphs}, the subgraph of $G$ induced by $B$ is $\b{C_5}\cong C_5$ and has
five edges. Also there is no edge of $G$ with both end-vertices in $\N(v)$, because $G$ is triangle free. Consequently,
the subgraph of $G$ induced by $\N[v]$ has exactly $n-6$ edges and all the remaining edges of $G$ has one end-vertex in
$A$ and one end-vertex in $B$. Assume that a vertex $x\in A$ is adjacent to at least three vertices $u_1,u_2,u_3$ in
$B$. Since $u_i$'s ($i=1,2,3$) lie on a 5-cycle, at least two of them, say $u_1$ and $u_2$, are adjacent and $xu_1u_2$
is a triangle in $G$, a contradiction. Therefore, each vertex of $A$ can be adjacent to at most two of the vertices in
$B$ and the number of edges between $A$ and $B$ is at most $2|A|=2(n-6)$. Hence $m\leq 2(n-6)+(n-6)+5$ and by applying
Proposition \ref{number of edges}, we deduce that $n^2-13n+38\leq 0$ which only holds for $n\leq 8$. We conclude that
$6\leq n\leq 8$.

If $n=6$, then $a=6$ and $b=0$, that is, $G$ consists of 6 vertices of degree one and is isomorphic to the graph (a) of
Figure \ref{fig}. If $n=7$, then $G$ has $a=5$ vertices of degree 2 and $b=2$ vertices of degree 1. This means that $G$
is the disjoint union of a 5-cycle and an edge, that is, $G$ is isomorphic to Figure \ref{fig}(b).

Now assume that $n=8$ and hence $m=10$ and $G$ has $a=4$ vertices of degree 3 and $b=4$ vertices of degree $2$. Let $v$
be a vertex of degree $3$. Then $G_v$ is $\b{C_4}$ and has two edges, say $xy$ and $uw$. Assume that
$\N(v)=\{v_1,v_2,v_3\}$. Since there is no edge between vertices of $\N(v)$ (because $G$ is triangle-free), the number
of edges with one end-vertex in $\N(v)$ and one end-vertex in $\V(G_v)$ is $m- \deg(v)-|\E(G_v)|=5$. Each of the four
vertices of $G_v$ has degree at least two in $G$ and exactly one in $G_v$. Thus each of them is adjacent to at least
one of the $v_i$'s, and exactly one of the vertices of $G_v$, say $x$, is adjacent to two of the $v_i$'s, say $v_1$ and
$v_3$ (see Figure \ref{fig}(c)). So $v_1$ and $v_3$ are not adjacent to $y$ (else $v_ixy$ is a triangle for $i=1$ or
$3$) and $y$ must be adjacent to $v_2$. According to Theorem \ref{goren graphs}, $G_x$ is also isomorphic to $\b{C_4}$.
But $G_x$ is the subgraph of $G$ induced by $\{u,w,v,v_2\}$ and we know that $uw,vv_2\in \E(G_x)$. Thus there is no
edge in $G$ with one end-vertex in $\{u,w\}$ and one end-vertex in $\{v,v_2\}$. In particular, $u$ and $w$ are not
adjacent to $v_2$ and one of them say $u$ must be adjacent to $v_1$ and the other must be adjacent to $v_3$.
Consequently, $G$ is the graph (c) in Figure \ref{fig}.
\end{proof}

\begin{rem}
Note that parts \eqref{tri-free-ind3 3} and \eqref{tri-free-ind3 4} of the above theorem, present a fully combinatorial
characterization of Gorenstein triangle-free graphs with $\alpha=3$. Also we mention that by this theorem, the only
Gorenstein triangle-free connected graph with independence number 3 is the graph (c) in Figure \ref{fig}.
\end{rem}

It is well-known that a graph is Gorenstein (over $K$) \ifof each of its connected components are Gorenstein (over
$K$). Thus the above remark poses the following question.
\begin{ques}
Is there an infinite family of connected Gorenstein graphs with $\alpha=3$?
\end{ques}
Note that if we remove connectedness from the above question, then the family $\{\b{C_n} \cup K_2| 4\leq n\in
\mathbb{N}\}$, is an easy answer to the question.



\begin{thebibliography}{5}

\bibitem{CM ring} W. Bruns and J\"urgen Herzog, \emph{Cohen-Macaulay Rings}, Cambridge University Press,
    Cambridge, 1993.

\bibitem{hibi} J. Herzog  and T. Hibi, \emph{Monomial Ideals}, Springer-Verlag, London , 2011.

\bibitem{large girth} D. T. Hoang, N. C. Minh and T. N. Trung, Cohen-Macaulay graphs with large girth, \emph{J.
    Algebra  Appl.} \textbf{14}:7, paper No. 1550112, 2015.
%
\bibitem{tri-free}D. T. Hoang and T. N. Trung, A characterization of triangle-free Gorenstein graphs and
    Cohen-Macaulayness of second powers of edge ideals, \emph{J. Algebr. Combin.} \textbf{43}, 325--338, 2016.

\bibitem{survey} V.E. Levit and E. Mandrescu, The independence polynomial of a graph --- a survey,
    \emph{Proceedings  of the 1st International Conference on Algebraic Informatics}, 233--254,
    Aristotle Univ. Thessaloniki, Thessaloniki, 2005.
%
\bibitem{my vdec} A. Nikseresht, Chordality of clutters with vertex decomposable dual and ascent of clutters,
    \emph{J. Combin. Theory ---Ser. A}, \textbf{168}, 318--337, 2019,  arXiv: 1708.07372.

\bibitem{trung} A. Nikseresht and M. R. Oboudi, Trung’s construction and the Charney-Davis conjecture, \emph{Bull.
    Malays. Math. Sci. Soc.}, 2020, available online: https://doi.org/10.1007/s40840-020-00933-8 .

\bibitem{our chordal} A. Nikseresht and R. Zaare-Nahandi, On generalization of cycles and chordality to clutters from
    an algebraic viewpoint, \emph{Algebra Colloq.} \textbf{24}:4, 611--624,  2017.

\bibitem{stanley} R. P. Stanley, \emph{Combinatorics and Commutative Algebra}, Birkh\"auser, Boston, 1996.

\bibitem{frontier} ---, Positivity problems and conjectures in algebraic combinatorics, in: V. Arnold,
    et al (Eds.), \emph{Mathematics: Frontiers and Prospectives}, American Mathematical Society, pp. 295--320, 2000.

\bibitem{W2} J. W. Staples, On some subclasses of well-covered graphs, \emph{J. Graph Theory} \textbf{3},
    197--204, 1979.

\bibitem{symbolic} N. V. Trung and T. M. Tuan, Equality of ordinary and symbolic powers of Stanley-Reisner ideals,
    \emph{J. Algebra} \textbf{328}, 77--93, 2011.

\bibitem{planar goren} T. N. Trung, A characterization of Gorenstein planar graphs, in T. Hibi, ed., \emph{Adv. Stud.
    Pure Math.}, Vol. 77 (2018): \emph{The 50th Anniversary of Gr\"obner Bases}, 399--409, arXiv:1603.00326v2.

\end{thebibliography}
\end{document}